\documentclass[11pt]{article}
\usepackage[margin=1in]{geometry}
\usepackage{amsmath,amssymb,amsthm,mathtools}
\usepackage{graphicx,booktabs,enumitem}
\usepackage[hidelinks]{hyperref}
\usepackage[numbers,sort&compress]{natbib}
\usepackage{authblk}
\newcommand{\InstMinRe}{0.129}
\newcommand{\IrArgminLam}{0.042}
\newcommand{\IrArgminQ}{1.076}
\newcommand{\IrBoxes}{2676}
\newcommand{\IrBoxesOne}{2293}
\newcommand{\IrBoxesTwo}{383}
\newcommand{\IrCoverOne}{99.40}
\newcommand{\IrMinLink}{0.057}
\newcommand{\IrMinRe}{6.7\cdot 10^{-4}}
\newcommand{\IrRollLinks}{2293}
\newcommand{\IrThreshR}{0.575}
\newcommand{\IrUncovered}{290}

\newcommand{\LwNpos}{7}
\newcommand{\OnsBoxes}{3520}
\newcommand{\OnsEpsMax}{0.55}
\newcommand{\OnsGluings}{6916}
\newcommand{\OnsLinkMargin}{0.031}
\newcommand{\OnsMaxRzero}{0.088}
\newcommand{\OnsMaxXnorm}{11}
\newcommand{\OnsMaxYzero}{0.064}
\newcommand{\OnsMaxZ}{0.038}
\newcommand{\OnsMinGap}{0.12}
\newcommand{\OnsMinGlue}{0.078}
\newcommand{\OnsMinNontriv}{3.8}

\newcommand{\PtsMaxRzero}{2.4\cdot 10^{-15}}

\newcommand{\PtsMinLink}{0.074}
\newcommand{\RegBoxes}{771}
\newcommand{\RegGluings}{1457}
\newcommand{\RegMaxRzero}{0.044}
\newcommand{\RegMaxYzero}{0.027}
\newcommand{\RegMaxZ}{0.53}
\newcommand{\RegMinGap}{0.029}
\newcommand{\RegMinGlue}{0.029}
\newcommand{\RegMinNontriv}{1.7}
\newcommand{\RegMinRuniq}{0.069}
\newcommand{\RegSplit}{7}
\newcommand{\IrFracPD}{67}

\newtheorem{theorem}{Theorem}[section]
\newtheorem{lemma}[theorem]{Lemma}
\newtheorem{proposition}[theorem]{Proposition}

\newtheorem{maintheorem}{Theorem}

\theoremstyle{definition}
\newtheorem{definition}[theorem]{Definition}
\newtheorem{remark}[theorem]{Remark}

\newcommand{\R}{\mathbb{R}}
\newcommand{\C}{\mathbb{C}}
\newcommand{\Z}{\mathbb{Z}}
\newcommand{\N}{\mathbb{N}}
\newcommand{\nrm}[1]{\lVert #1\rVert}
\newcommand{\abs}[1]{\lvert #1\rvert}
\newcommand{\lnu}{\ell^1_\nu}
\renewcommand{\Re}{\operatorname{Re}}
\renewcommand{\Im}{\operatorname{Im}}

\title{Existence and spectral instability of steady rolls in the Nikolaevskiy equation: a computer-assisted study}
\author[1]{Nikita V. Konev\thanks{Corresponding author.}}
\author[2]{Fedor A. Vyazanitsyn}
\author[1]{Igor A. Knyazev}
\affil[1]{Faculty of Physics, ITMO~University, Saint Petersburg 197101, Russia}
\affil[2]{Computer Engineering Laboratory, Faculty of Software Engineering and Computer Systems,
ITMO~University, Saint Petersburg 197101, Russia}
\affil[ ]{\texttt{nvkonev@itmo.ru}, \texttt{faviazanitcyn@itmo.ru}, \texttt{iakniazev@itmo.ru}}
\date{}

\begin{document}
\maketitle

\begin{abstract}
The Nikolaevskiy equation $u_t+uu_x=-\partial_x^2\bigl(ru-(1+\partial_x^2)^2u\bigr)$ is a paradigm for
short-wave pattern formation coupled to a neutral long-wave (Galilean) mode, and it is believed that all
of its spatially periodic steady states are unstable. Existing results are formal (amplitude equations),
numerical, or restricted to a small neighbourhood of onset. We give computer-assisted proofs of the
following. (A) For every wavenumber $q\in[0.9,1.1]$ and every control parameter $r$ with
$(1-q^2)^2<r\le 4$, the equation has a nontrivial, real-analytic, odd steady roll of period $2\pi/q$; these
rolls form a single continuous two-parameter family which bifurcates from the trivial state at the onset
curve, with $\nrm{u}_\infty\le 2\cdot\OnsMaxXnorm\sqrt{r-(1-q^2)^2}$ near onset. (B) Every roll of this family with
$r\in[0.25,4]$ is spectrally unstable in $L^2(\R)$: its linearisation has a Bloch eigenvalue $\Lambda$ with
$\Re\Lambda\ge \IrMinRe$, and $\Re\Lambda\ge0.13$ for $r\ge \IrThreshR$. (C) At nine representative
parameter values we enclose, at the Bloch number $\sigma=q/100$, the eigenvalue numerically
identified with the long-wave (Matthews--Cox) branch; its real part is positive at seven points and negative
at two, while a Bloch eigenvalue with $\Re\Lambda\ge\InstMinRe$ exists at all nine. The proofs combine a parametrised
Newton--Kantorovich theorem, a desingularisation near onset, Taylor-model bounds on parameter boxes, and a
Weyl-sequence argument for the Bloch spectrum; all estimates are evaluated in rigorous interval arithmetic.
\end{abstract}

{\sloppy\noindent\textbf{Keywords:} Nikolaevskiy equation; pattern formation; computer-assisted proof;
Newton--Kantorovich; validated continuation; spectral instability; Bloch waves.\par}


\section{Introduction}

The Nikolaevskiy equation
\begin{equation}\label{eq:nik}
  u_t + u u_x = -\partial_x^2\bigl(r u - (1+\partial_x^2)^2 u\bigr), \qquad x\in\R,\ t\ge 0,
\end{equation}
was proposed by Beresnev and Nikolaevskiy as a model for seismic waves in viscoelastic media
\cite{BeresnevNikolaevskiy1993}. Its linear growth rate $\sigma(k)=k^2\bigl(r-(1-k^2)^2\bigr)$ combines a
Turing-type band of unstable wavenumbers around $k=1$ for $r>0$ with a neutral mode at $k=0$ inherited from
Galilean invariance. Tribelsky and Velarde \cite{TribelskyVelarde1996} and Tribelsky and Tsuboi
\cite{TribelskyTsuboi1996} argued that this coupling destabilises all spatially periodic steady states and
leads to ``soft-mode turbulence'' directly from onset. Matthews and Cox \cite{MatthewsCox2000} derived the
asymptotically consistent amplitude equations, which predict that all rolls are unstable near onset, and
Cox and Matthews \cite{CoxMatthews2007} confirmed this picture numerically and studied a damped variant.
Zimmermann \cite{Zimmermann2017} proved that the Matthews--Cox system correctly approximates the dynamics for
$0<r\ll1$ on long time scales, and Sukhtayev \cite{Sukhtayev2016} analysed diffusive stability of
small-amplitude patterns with a conservation law. Kudryashov \cite{Kudryashov2003} constructed closed-form periodic
travelling waves at the special value $r=81/121$; see Remark~\ref{rem:exact}.

This makes \eqref{eq:nik} qualitatively different from the Swift--Hohenberg equation. There, rolls
bifurcating at a Turing instability are stable inside the Eckhaus region near onset \cite{Mielke1997}. In
\eqref{eq:nik} the neutral mode at $k=0$ couples to the pattern and removes this stable band near onset, in
the Matthews--Cox regime. Whether nonlinear saturation at finite amplitude, $r=O(1)$, can restore a band of
stable rolls has remained open. Theorem~\ref{thm:B} below answers it negatively on the rectangle
$[0.9,1.1]\times[0.25,4]$: no roll of the family is spectrally stable there. Theorem~\ref{thm:C} concerns the
mechanism: at two of nine sample points the eigenvalue numerically identified with the long-wave branch has
negative real part at $\sigma=q/100$, whereas eigenvalues at larger Bloch numbers are certified unstable.

We are not aware of rigorous results on finite-amplitude rolls of \eqref{eq:nik}, neither on their
existence away from onset nor on their stability. In this paper we provide such results by computer-assisted
proof. The methodology builds on rigorous numerics for dissipative PDEs, in particular the
Kuramoto--Sivashinsky equation \cite{ZgliczynskiMischaikow2001,ArioliKoch2010,KubicaZgliczynskiKalita2025},
validated continuation \cite{DayLessardMischaikow2007,GameiroLessardMischaikow2008}, and recent
computer-assisted analyses of patterns and their spectra
\cite{CadiotLessardNave2025,vanderAalstEtAl2025,BlancoCadiot2026,BredenCadiotZurek2026,BlancoLessard2026,AyalaGarciaAzpeitiaLessard2025}.

\paragraph{Main results.} Write $r_c(q)=(1-q^2)^2$ for the onset curve of rolls with wavenumber $q$, and let
$\nu=1.05$ and $\lnu$ be the weighted sequence space of Section~\ref{sec:setting}.

\begin{maintheorem}[Existence down to onset]\label{thm:A}
Let $\mathcal D=\{(q,r): q\in[0.9,1.1],\ r_c(q)<r\le 4\}$. There is a continuous map
$\mathcal D\ni(q,r)\mapsto a^*(q,r)\in\lnu$ such that, for every $(q,r)\in\mathcal D$,
\[
  u_{q,r}(x)=\sum_{k\ge1}2a^*_k(q,r)\sin(kqx)
\]
is a real-analytic, odd, $2\pi/q$-periodic steady solution of \eqref{eq:nik} with $a^*_1(q,r)>0$, so its
minimal period is $2\pi/q$. Each $a^*(q,r)$ is locally unique in $\lnu$, that is, among odd
solutions: it is the only zero of $F(\cdot\,;q,r)$ in the closed ball of radius $\rho(q,r)$ about it, where
$\rho\ge\RegMinGap$ for $r\ge0.25$ and $\rho\ge\OnsMinGap\,(r-r_c(q))$ for $r-r_c(q)\le\OnsEpsMax^2$. Moreover
$\nrm{a^*(q,r)}_\nu\le \OnsMaxXnorm\sqrt{r-r_c(q)}$ whenever $r-r_c(q)\le \OnsEpsMax^2$. In particular the
family bifurcates from $u\equiv0$ exactly at $r=r_c(q)$. The half-period shift $x\mapsto x+\pi/q$
preserves odd solutions and maps $a_k\mapsto(-1)^ka_k$, so $u_{q,r}(\cdot+\pi/q)$, with first coefficient
$-a^*_1<0$, is a second family of odd rolls, exchanged with the first by this symmetry.
\end{maintheorem}

\begin{maintheorem}[Spectral instability]\label{thm:B}
For every $(q,r)\in[0.9,1.1]\times[0.25,4]$ the linearisation of \eqref{eq:nik} about $u_{q,r}$,
\[
  \mathcal L_{q,r}v=-\partial_x^2\bigl(rv-(1+\partial_x^2)^2v\bigr)-\partial_x(u_{q,r}v),
  \qquad \mathcal L_{q,r}\colon H^6(\R)\subset L^2(\R)\to L^2(\R),
\]
has $L^2$-spectrum intersecting the half-plane $\{\Re\Lambda\ge \IrMinRe\}$, and intersecting $\{\Re\Lambda\ge0.13\}$ if $r\ge\IrThreshR$. More precisely, there are $\sigma\in(0,q/2]$, $\Lambda\in\C$ and a nonzero
$2\pi/q$-periodic analytic function $P$ such that $v=e^{i\sigma x}P$ satisfies
$\mathcal L_{q,r}v=\Lambda v$ with these bounds on $\Re\Lambda$. Hence $e^{\Lambda t}v$ is a bounded-in-space
solution of the linearised equation that grows exponentially in time.
\end{maintheorem}

\begin{maintheorem}[Bloch eigenvalues at sample points]\label{thm:C}
At the nine points $(q,r)\in\{0.9,1,1.1\}\times\{0.25,1,4\}$ the data in Table~\ref{tab:points} hold. In
particular, at every point there is a Bloch eigenvalue with $\Re\Lambda\ge\InstMinRe$. Moreover, at
$\sigma=q/100$ there is an eigenvalue $\Lambda_{\rm lw}$ of \eqref{eq:bloch} with $\Re\Lambda_{\rm lw}>0$ at
$\LwNpos$ points and $\Re\Lambda_{\rm lw}<0$ at $(q,r)=(0.9,4)$ and $(1.1,0.25)$.
\end{maintheorem}

\begin{remark}[What Theorem~\ref{thm:C} does not assert]\label{rem:lw}
The eigenvalue $\Lambda_{\rm lw}$ is selected as one of the two eigenvalues of smallest modulus of the
truncated operator, and numerically it lies on the branch that emanates from the double zero eigenvalue at
$\sigma=0$ (translation and Galilean modes, Section~\ref{sec:spectral}). This identification is not part of
the theorem; proving it would require enclosing the two-dimensional spectral subspace near $0$ uniformly in
$\sigma\in(0,q/100]$. Nor does $\Re\Lambda_{\rm lw}<0$ at one Bloch number exclude long-wave instability at
smaller ones.
\end{remark}

\begin{remark}[On closed-form solutions]\label{rem:exact}
Kudryashov \cite{Kudryashov2003} constructs elliptic travelling waves of
$u_t+uu_x+\beta u_{xx}+\delta u_{xxxx}+\varepsilon u_{xxxxxx}=0$ under the constraint
$\beta=10\delta^2/(121\varepsilon)$. For \eqref{eq:nik} we have $\beta=r-1$, $\delta=-2$, $\varepsilon=-1$, so
the constraint singles out $r=81/121$. Kamenov \cite{Kamenov2014} uses the ansatz
$u=a+\mu(\ln\theta_4)_{xx}$, i.e.\ $u=A+B\,\wp(kx+\omega t-z_0)$ with Weierstrass $\wp$. Such a function
has double poles. When the Laurent expansion $\wp(z)=z^{-2}+O(z^2)$ is substituted into \eqref{eq:nik}, the
coefficient of $z^{-8}$ is $-5040Bk^6$, and no other term contributes at this order. Hence such an ansatz
solves \eqref{eq:nik} only if $Bk=0$. This is consistent with the fact that the dominant balance between
$uu_x$ and $\partial_x^6u$ requires fifth-order poles.
\end{remark}

\paragraph{Outline.} Section~\ref{sec:abstract} contains the abstract fixed-point and gluing results.
Section~\ref{sec:setting} sets up the Galerkin formulation and proves that zeros are classical solutions.
Section~\ref{sec:boxes} gives the bounds on parameter boxes. Section~\ref{sec:onset} treats the onset region
by desingularisation. Section~\ref{sec:spectral} proves the spectral results. Section~\ref{sec:impl}
describes the implementation and its verification, and Section~\ref{sec:disc} discusses the results.

\section{A parametrised Newton--Kantorovich theorem and gluing}\label{sec:abstract}

\begin{theorem}\label{thm:NK}
Let $X$ be a Banach space, $\Pi$ a compact metric space, and for $p\in\Pi$ let $T_p\colon X\to X$ be
$C^1$. Assume that $(p,x)\mapsto T_p(x)$ and $(p,x)\mapsto DT_p(x)$ are continuous, that $p\mapsto\bar x_p\in X$
is continuous, and that there are constants $Y_0,Z\ge0$ and $Z_2,r_*>0$ with, for all $p\in\Pi$,
\[
  \nrm{T_p(\bar x_p)-\bar x_p}\le Y_0,\qquad \nrm{DT_p(\bar x_p)}\le Z,\qquad
  \nrm{DT_p(\bar x_p+h)-DT_p(\bar x_p)}\le Z_2\nrm h\quad(\nrm h\le r_*).
\]
Let $P(\rho)=Z_2\rho^2-(1-Z)\rho+Y_0$. If $0<r_0\le r_1\le r_*$ and $P(r_0)<0$, $P(r_1)<0$, then for every
$p\in\Pi$ the map $T_p$ has a unique fixed point $x^*_p$ in $\bar B_{r_1}(\bar x_p)$, this fixed point lies in
$\bar B_{r_0}(\bar x_p)$, and $p\mapsto x^*_p$ is continuous.
\end{theorem}

\begin{proof}
Let $\rho\in\{r_0,r_1\}$ and $x\in\bar B_\rho(\bar x_p)$. By the mean value inequality,
$\nrm{DT_p(x)}\le Z+Z_2\rho$ and
\[
  \nrm{T_p(x)-\bar x_p}\le \nrm{T_p(\bar x_p)-\bar x_p}+(Z+Z_2\rho)\rho\le P(\rho)+\rho<\rho .
\]
Moreover $P(\rho)<0$ gives $Z+Z_2\rho<1-Y_0/\rho\le1$. So $T_p$ is a contraction of $\bar B_\rho(\bar x_p)$
into itself. The fixed point in $\bar B_{r_0}$ is also a fixed point in $\bar B_{r_1}$, and the latter is
unique. For continuity, set $S_p(y)=T_p(\bar x_p+y)-\bar x_p$. This is a family of contractions of the fixed
ball $\bar B_{r_0}(0)$ with a common constant $Z+Z_2r_0<1$, continuous in $p$. By the uniform contraction
principle its fixed point $y^*_p$ is continuous, and so is $x^*_p=\bar x_p+y^*_p$.
\end{proof}

In all applications $T_p(x)=x-A_pF_p(x)$ with $A_p$ injective, so fixed points of $T_p$ are exactly zeros of
$F_p$. The map $F_p$ itself does not act on $X$: its linear part grows like $k^6$. On the infinite tail,
however, $A_p$ is the exact inverse of that linear part at the current parameter value: $\lambda_k(q,r)^{-1}$
in Sections~\ref{sec:setting}--\ref{sec:onset}, and $(g(\kappa_n)-\bar\Lambda_p)^{-1}$ in
Section~\ref{sec:spectral}, with $\bar\Lambda_p$ the eigenvalue component of the centre $\bar x_p$. The
unbounded part therefore cancels identically, e.g.\ $(T_p(a))_k=-\frac{qk}{2\lambda_k}(a*a)_k$ for $k>N$,
and $T_p\colon X\to X$ is a polynomial map, continuous in $p$. When we call an approximate inverse
``constant on a box'', only its finite block is meant. The bound on $\nrm{DT_p(\bar x_p)}$ is split as $Z=Z_0+Z_1$ in the usual way: $Z_0$ measures the
finite-dimensional defect of $A_p$ and $Z_1$ the coupling to the tail.

\begin{definition}\label{def:family}
A \emph{locally unique family} of zeros of $(F_p)_{p\in\Pi}$ is a continuous map $p\mapsto x^*_p$ with
$F_p(x^*_p)=0$, together with a finite cover of $\Pi$ by closed sets $\Pi^{(\ell)}$, continuous centres
$\bar x^{(\ell)}$ on $\Pi^{(\ell)}$ and radii $r_0^{(\ell)}<r_1^{(\ell)}$, such that for $p\in\Pi^{(\ell)}$ the
point $x^*_p$ lies in $\bar B_{r_0^{(\ell)}}(\bar x^{(\ell)}_p)$ and is the only zero of $F_p$ in
$\bar B_{r_1^{(\ell)}}(\bar x^{(\ell)}_p)$.
\end{definition}

Theorem~\ref{thm:NK} provides such a family on a single box. Two such families on boxes sharing a face, which
coincide on the face (Lemma~\ref{lem:glue} below), together form a locally unique family on the union.

\begin{lemma}[Gluing]\label{lem:glue}
Let $x^{*,1}$ be a locally unique family on $\Pi_1$ in the sense of Definition~\ref{def:family}. Let
$p\mapsto x^{*,2}_p$ be any continuous family of zeros of $F_p$ on $\Pi_2$, with the same $F_p$ on
$S=\Pi_1\cap\Pi_2$. Assume that $S$ is connected and that, at one point $p_0\in S$ and for one $\ell$ with
$p_0\in\Pi_1^{(\ell)}$, there are $\bar y$ and $\rho$ with $x^{*,2}_{p_0}\in\bar B_\rho(\bar y)$ and
$\nrm{\bar y-\bar x^{1,(\ell)}_{p_0}}+\rho\le r_1^{1,(\ell)}$. Then $x^{*,1}_p=x^{*,2}_p$ for all $p\in S$.
\end{lemma}

\begin{proof}
Let $E=\{p\in S: x^{*,1}_p=x^{*,2}_p\}$.
\begin{itemize}[leftmargin=*]
\item The set $E$ is closed, because both families are continuous.
\item The point $p_0$ lies in $E$: by hypothesis $x^{*,2}_{p_0}\in\bar B_{r_1^{1,(\ell)}}(\bar x^{1,(\ell)}_{p_0})$, where the
only zero of $F_{p_0}$ is $x^{*,1}_{p_0}$.
\item The set $E$ is open in $S$. Let $p\in E$. Only finitely many cover sets $\Pi_1^{(\ell)}$ contain $p$, and the
union of the others is closed and misses $p$. Hence every $p'\in S$ close enough to $p$ lies in some
$\Pi_1^{(\ell)}\ni p$. For such $\ell$,
\[
  \nrm{x^{*,2}_{p'}-\bar x^{1,(\ell)}_{p'}}\le\nrm{x^{*,2}_{p'}-x^{*,2}_p}+\nrm{x^{*,1}_p-\bar x^{1,(\ell)}_p}
  +\nrm{\bar x^{1,(\ell)}_p-\bar x^{1,(\ell)}_{p'}}\le o(1)+r_0^{1,(\ell)}<r_1^{1,(\ell)},
\]
so $x^{*,2}_{p'}=x^{*,1}_{p'}$.
\end{itemize}
Since $S$ is connected, $E=S$.
\end{proof}

Uniqueness is required only of the first family. The second family needs only an existence enclosure at
one point. For two boxes sharing an edge, both from Theorem~\ref{thm:NK} and with centres affine in $p$, we
check $\nrm{\bar x^{i}_p-\bar x^{j}_p}+r_0^{i}<r_1^{j}$ at the two endpoints of the edge. The left-hand side
is a convex function of $p$ along the edge, so the inequality holds on the whole edge, and the zeros agree
there directly. When a box meets two smaller boxes along one side (a hanging node), the shared boundary
consists of two segments. Both centres are affine on each of them, and the check is made at the endpoints of
each segment.

\section{Galerkin formulation of steady rolls}\label{sec:setting}

\begin{figure}[t]
\centering\includegraphics[width=0.85\linewidth]{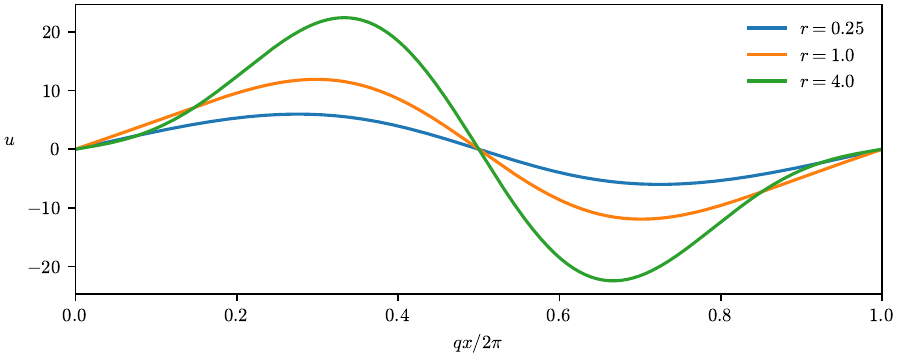}
\caption{Centres $\bar a$ of the certified rolls at $q=1$, plotted over one period. The true solution is
within $\PtsMaxRzero$ of the centre in $\lnu$.}
\label{fig:rolls}
\end{figure}

Fix $q>0$ and set $y=qx$. An odd $2\pi/q$-periodic function is written $u=\sum_{k\ge1}2a_k\sin(ky)$. Extend
$a$ oddly, $a_{-k}=-a_k$, $a_0=0$, and write $(a*b)_k=\sum_{m\in\Z}a_mb_{k-m}$. A direct computation
(verified symbolically, see Section~\ref{sec:impl}) shows that the $\sin(ky)$-coefficient of the steady
residual
\[
  uu_x+\partial_x^2\bigl(ru-(1+\partial_x^2)^2u\bigr)
\]
equals $2F_k(a)$, where
\begin{equation}\label{eq:F}
  F_k(a;q,r)=\lambda_k a_k+\frac{qk}{2}(a*a)_k,\qquad \lambda_k=q^2k^2\bigl((q^2k^2-1)^2-r\bigr),\quad k\ge1 .
\end{equation}
Let $\nu>1$ and $\lnu=\{a=(a_k)_{k\ge1}\in\R^\N:\nrm a_\nu=\sum_{k\ge1}\abs{a_k}\nu^k<\infty\}$; the spectral
problem of Section~\ref{sec:spectral} uses the complex two-sided space $\ell^1_\nu(\Z)$. On finite matrices
we use the induced norm $\nrm M=\max_j\nu^{-j}\sum_i\abs{M_{ij}}\nu^i$.

\begin{lemma}\label{lem:algebra}
For $a,b\in\lnu$, $\sum_{k\ge1}\abs{(a*b)_k}\nu^k\le2\nrm a_\nu\nrm b_\nu$.
\end{lemma}

\begin{proof}
The two-sided weighted norm $\sum_{k\in\Z}\abs{c_k}\nu^{\abs k}$ is submultiplicative under convolution,
because $\abs k\le\abs m+\abs{k-m}$. For odd extensions it equals $2\nrm{\cdot}_\nu$. The product $a*b$ is
even, so its $k\ge1$ part has at most half of its two-sided norm: $\tfrac12\cdot2\nrm a_\nu\cdot2\nrm b_\nu$.
\end{proof}

\begin{lemma}\label{lem:classical}
If $a\in\lnu$ with $\nu>1$ and $F(a;q,r)=0$, then $u=\sum2a_k\sin(kqx)$ is real analytic in the strip
$\abs{\Im x}<\ln\nu/q$ and is a classical steady solution of \eqref{eq:nik}.
\end{lemma}

\begin{proof}
We have $\abs{a_k}\le\nrm a_\nu\nu^{-k}$, so the series and all its derivatives converge uniformly on
compact subsets of the strip. The residual $R(u)$ is therefore an odd, analytic, periodic function. Its
Fourier coefficients are obtained by termwise differentiation and the Cauchy product, which are justified
by absolute convergence. They equal $2F_k(a)=0$, hence $R(u)\equiv0$.
\end{proof}

\begin{lemma}[Tail]\label{lem:tail}
Let $N\ge1$. If $q^2(N+1)^2>1+\sqrt r$, then $\lambda_k>0$ for $k\ge N+1$, and $k/\lambda_k$ is decreasing in
$k\ge N+1$. Hence $\sup_{k>N}k/\lambda_k=(N+1)/\lambda_{N+1}=:\mu_N$.
\end{lemma}

\begin{proof}
Write $g(k)=k\bigl((q^2k^2-1)^2-r\bigr)$ and $s=q^2k^2-1$. Then $k/\lambda_k=1/(q^2g(k))$ and
$g'(k)=s^2-r+4s(s+1)$, which is positive for $s>\sqrt r$.
\end{proof}

\subsection{Bounds at a point}\label{sec:pointbounds}

Let $\bar a\in\lnu$ have $\bar a_k=0$ for $k>N$. Let $J_N$ be the $N\times N$ Jacobian of
$(F_1,\dots,F_N)$ at $\bar a$, i.e.\ $(J_N)_{kj}=\lambda_k\delta_{kj}+qk(\bar a_{k-j}-\bar a_{k+j})$, and
let $A_N\approx J_N^{-1}$. Set $A=\operatorname{diag}(A_N,(\lambda_k^{-1})_{k>N})$, $T(a)=a-AF(a)$, and
$K=\operatorname{diag}(k)$.

\begin{proposition}\label{prop:pointbounds}
Under the hypothesis of Lemma~\ref{lem:tail}, the bounds required by Theorem~\ref{thm:NK} (with $\Pi$ a
point) hold for $T$ on $\lnu$, with $r_*=\infty$ and
\begin{align*}
  Y_0&=\nrm{A_NF_{\le N}(\bar a)}_\nu+\sum_{k=N+1}^{2N}\frac{\abs{F_k(\bar a)}}{\lambda_k}\nu^k,\qquad
  Z_0=\nrm{I-A_NJ_N},\\
  Z_1&=\bigl\|\,\abs{A_N}\zeta\,\bigr\|_\nu+2q\mu_N\nrm{\bar a}_\nu,\qquad
  \zeta_k=qk\max_{N-k<j\le N}\abs{\bar a_j}\nu^{-(k+j)},\\
  Z_2&=2q\max\bigl(\nrm{A_NK_N},\mu_N\bigr).
\end{align*}
Its conclusion follows whenever $P(r_0)<0$ and $P(r_1)<0$ for some $0<r_0\le r_1$.
\end{proposition}

\begin{proof}
The residual $F(\bar a)$ is supported in $k\le2N$, which gives $Y_0$. Write
\[
  I-A\,DF(\bar a)=(I-AA^\dagger)+A\bigl(A^\dagger-DF(\bar a)\bigr),\qquad
  A^\dagger=\operatorname{diag}\bigl(J_N,(\lambda_k)_{k>N}\bigr).
\] The first term is $I-A_NJ_N$ on the first $N$
modes and zero on the tail.
On rows $k\le N$, $(A^\dagger-DF(\bar a))h=-qk(\bar a*h_{\rm tail})_k$. Since $\bar a$ has support in
$\abs m\le N$, only $m=-j$ with $j>N-k$ contributes, so $\abs{qk(\bar a*h_{\rm tail})_k}\le\zeta_k\nrm h_\nu$.
On rows $k>N$, $(A(A^\dagger-DF(\bar a))h)_k=-(qk/\lambda_k)(\bar a*h)_k$, and Lemmas~\ref{lem:algebra}
and~\ref{lem:tail} bound its norm by $2q\mu_N\nrm{\bar a}\nrm h$.
Finally $DF(\bar a+h)-DF(\bar a)=qK(h*\cdot)$, so $\nrm{A(DF(\bar a+h)-DF(\bar a))}\le q\nrm{AK}\,2\nrm h$, and
$\nrm{AK}=\max(\nrm{A_NK_N},\mu_N)$. Injectivity of $A$ follows from $Z_0<1$.
\end{proof}

\section{Parameter boxes}\label{sec:boxes}

Let $\Pi=[q_m-h_q,q_m+h_q]\times[r_m-h_r,r_m+h_r]$ and write $q=q_m+s$, $r=r_m+t$. We use an affine centre
$\bar a(s,t)=a^0+s\,a^q+t\,a^r$, where $a^q,a^r$ are numerical tangents. Every quantity is then a polynomial in
$(s,t)$: $\lambda_k$ has degree $6$ in $s$ and $1$ in $t$, and $F$ is quadratic in $a$. Its coefficients
$c_{ij}$ are computed in interval arithmetic and bounded on $\Pi$ by
\begin{equation}\label{eq:taylor}
  \Bigl|\sum_{i,j}c_{ij}s^it^j\Bigr|\le\sum_{i,j}\abs{c_{ij}}h_q^ih_r^j .
\end{equation}
The coefficients of order zero and one in $F(\bar a(s,t);s,t)$ are the defects
\[
  \eta=F(a^0;q_m,r_m),\qquad \eta^p=J\,a^p+\partial_pF(a^0;q_m,r_m)\quad(p=q,r)
\]
of the centre and of the tangents. They are small but not zero, and \eqref{eq:taylor} bounds them together with
all higher coefficients, so that, with $A_p$ the approximate inverse at $p\in\Pi$ (its tail depends on $p$),
\[
  Y_0\le\sup_{p\in\Pi}\nrm{A_p\eta}+h_q\sup_{p\in\Pi}\nrm{A_p\eta^q}+h_r\sup_{p\in\Pi}\nrm{A_p\eta^r}+Ch^2 ,
\]
where $h=\max(h_q,h_r)$ and $C$ bounds the terms of order two and higher uniformly on $\Pi$.
Likewise
\[
  \nrm{I-A_NJ_N(s,t)}\le\nrm{I-A_NJ_N^{(0,0)}}+\sum_{(i,j)\ne(0,0)}\nrm{A_NJ_N^{(i,j)}}h_q^ih_r^j .
\]
The bounds $Z_1$, $Z_2$ and the tail are those of Proposition~\ref{prop:pointbounds}, with $q,r$ replaced by
intervals and $\abs{\bar a_j}$ by $\abs{a^0_j}+h_q\abs{a^q_j}+h_r\abs{a^r_j}$. For $q^2k^2>1+\sqrt r$,
$\lambda_k$ is increasing in $q$ and decreasing in $r$, as $\partial_s[s((s-1)^2-r)]>0$ for $s=q^2k^2$.
Hence the worst tail constant $\mu_N$ on a box occurs at the corner $(q_m-h_q,r_m+h_r)$. We evaluate it in
interval arithmetic, which encloses this value. Nontriviality,
$a^*_1\ne0$ on $\Pi$, follows from $\abs{a^0_1}-h_q\abs{a^q_1}-h_r\abs{a^r_1}>r_0$.

\begin{proposition}\label{prop:region}
The rectangle $[0.9,1.1]\times[0.25,4]$ is covered by $\RegBoxes$ closed boxes. Of the 750 boxes of an initial
$10\times75$ grid, $\RegSplit$ were subdivided once into four congruent sub-boxes. On every box Theorem~\ref{thm:NK} holds with
$Y_0\le\RegMaxYzero$, $Z_0+Z_1\le\RegMaxZ$, $r_0\le\RegMaxRzero$ and $r_1\ge\RegMinRuniq$, and the
nontriviality margin is at least $\RegMinNontriv$. All $\RegGluings$ shared edges satisfy the gluing
inequality, with margin at least $\RegMinGlue$, and the adjacency graph is connected. Each of the nine point
proofs of Table~\ref{tab:points} lies in the uniqueness ball of a region box containing its $(q,r)$, with
margin at least $\PtsMinLink$; hence the rolls of Table~\ref{tab:points} belong to this family.
\end{proposition}

\section{The onset region: desingularisation}\label{sec:onset}

\begin{figure}[t]
\centering\includegraphics[width=0.85\linewidth]{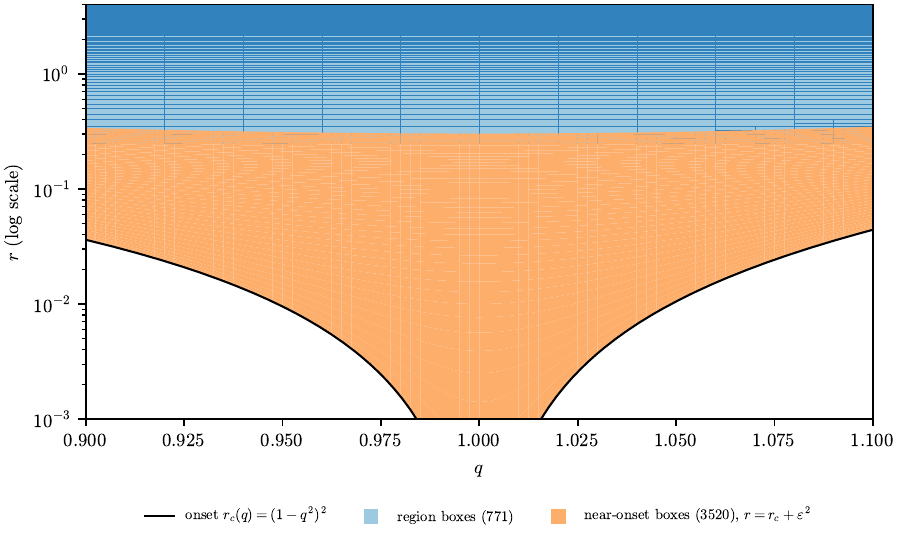}
\caption{Certified covering of the existence region $\mathcal D$. Blue: parameter boxes of
Proposition~\ref{prop:region}. Orange: boxes of Proposition~\ref{prop:onset} mapped to $(q,r)$ through
$r=r_c(q)+\varepsilon^2$. The black curve is the onset curve $r_c(q)$.}
\label{fig:cover}
\end{figure}

As $r\downarrow r_c(q)$ the rolls shrink to zero and $J_N$ degenerates, so the bounds of
Section~\ref{sec:boxes} cannot reach the onset curve. We blow up the singularity.

\begin{lemma}\label{lem:blowup}
Let $r=r_c(q)+\varepsilon^2$ with $0<\varepsilon\le1$, and set $a=S(\varepsilon)x=\varepsilon c(x)$ with
$c(x)=(x_1,\varepsilon x_2,\varepsilon x_3,\dots)$, that is, $S(\varepsilon)x=(\varepsilon x_1,\varepsilon^2x_2,\varepsilon^2x_3,\dots)$. Then $F_1(a)=\varepsilon^3G_1(x)$ and
$F_k(a)=\varepsilon^2G_k(x)$ for $k\ge2$, where
\begin{align*}
  G_1(x)&=-q^2x_1-q\Bigl(x_1x_2+\varepsilon\sum_{m\ge2}x_mx_{m+1}\Bigr),\\
  G_k(x)&=\lambda_kx_k+\frac{qk}{2}\bigl(c(x)*c(x)\bigr)_k,\qquad k\ge2 .
\end{align*}
The map $G$ is polynomial in $(x,q,\varepsilon)$ and extends to $\varepsilon=0$.
\end{lemma}

\begin{proof}
We have $\lambda_1=q^2(r_c-r)=-q^2\varepsilon^2$ and $(c*c)_1=-2\sum_{m\ge1}c_mc_{m+1}$, which is divisible
by $\varepsilon$. The rest is direct substitution.
\end{proof}

\begin{lemma}\label{lem:nondeg}
Let $q\in[0.9,1.1]$ and $\varepsilon=0$. The sequence $x^0=(\sqrt{\lambda_2},-q,0,0,\dots)$ solves $G=0$, where
\[
  \lambda_2=\lambda_2\bigl(q,r_c(q)\bigr)=4q^2\bigl((4q^2-1)^2-r_c(q)\bigr).
\]
Let $Y_\nu=\{b:\sum_{k\ge1}\abs{b_k}\nu^kk^{-6}<\infty\}$. Then $G\colon\lnu\to Y_\nu$ is $C^1$, and
$DG(x^0)\colon\lnu\to Y_\nu$ is a bounded isomorphism: its $2\times2$ block in $(x_1,x_2)$ has determinant
$2q^2\lambda_2>0$, and in the remaining variables it is diagonal with entries $\lambda_k$, $k\ge3$, which
satisfy $c_0k^6\le\lambda_k\le C_0k^6$ with $c_0>0$. Here $\lambda_2=12q^4(5q^2-2)\ge16.1$.
\end{lemma}

The bounds for $G$ follow the same pattern as in Proposition~\ref{prop:pointbounds}, with $\bar a$ replaced by
$c(\bar x)=(\bar x_1,\varepsilon\bar x_2,\dots)$. This applies in $\zeta_k$ for $2\le k\le N$, where the tail
enters through $c(h)_j=\varepsilon h_j$, giving
$\zeta_k=qk\varepsilon\max_{N-k<j\le N}\abs{c(\bar x)_j}\nu^{-(k+j)}$, and in the tail term
$2q\mu_N\nrm{c(\bar x)}_\nu$, using $\nrm{c(h)}_\nu\le\nrm h_\nu$ for $\varepsilon\le1$. Two further changes
are needed. The tail coupling of row $1$ is $\zeta_1=q\varepsilon\abs{\bar x_N}\nu^{-(N+1)}$. The second derivative is
$D^2G=2B$ with $B_1(h,h')=-\tfrac q2\bigl[h_1h_2'+h_2h_1'+\varepsilon\sum_{m\ge2}(h_mh'_{m+1}+h'_mh_{m+1})\bigr]$
and $B_k(h,h')=\tfrac{qk}{2}(c(h)*c(h'))_k$, which yields
$Z_2=q\bigl(2+2\nu^{-2}+2\varepsilon\nu^{-4}\bigr)\max(\nrm{A_NK_N},\mu_N)$.

\begin{proposition}\label{prop:onset}
The rectangle $(q,\varepsilon)\in[0.9,1.1]\times[0,\OnsEpsMax]$ is covered by $\OnsBoxes$ boxes on which
Theorem~\ref{thm:NK} holds for $G$, with $Y_0\le\OnsMaxYzero$, $Z_0+Z_1\le\OnsMaxZ$ and $r_0\le\OnsMaxRzero$,
and $\abs{\bar x_1}-r_0\ge\OnsMinNontriv$. All $\OnsGluings$ gluing inequalities hold, with margin at least
$\OnsMinGlue$, and the adjacency graph is connected. At $(q,r)=(1,0.28)$, which lies in both the onset family
and the family of Proposition~\ref{prop:region}, the inequality of Lemma~\ref{lem:glue} holds with margin
$\OnsLinkMargin$.
\end{proposition}

\begin{lemma}[Uniqueness in original coordinates]\label{lem:uniqonset}
Let $0<\varepsilon\le1$, and let $x^*$ be the only zero of $G$ in $\bar B_{r_1}(\bar x)$, with
$\nrm{x^*-\bar x}_\nu\le r_0<r_1$. Then $a^*=S(\varepsilon)x^*$ is the only zero of $F(\cdot\,;q,r_c(q)+\varepsilon^2)$
in $\bar B_\rho(a^*)$ with $\rho=\varepsilon^2(r_1-r_0)$.
\end{lemma}

\begin{proof}
The map $S(\varepsilon)$ is invertible on $\lnu$, with $(S(\varepsilon)^{-1}h)_1=h_1/\varepsilon$ and
$(S(\varepsilon)^{-1}h)_k=h_k/\varepsilon^2$ for $k\ge2$, so $\nrm{S(\varepsilon)^{-1}}\le\varepsilon^{-2}$. Let
$F(a)=0$ and $\nrm{a-a^*}_\nu\le\rho$. By Lemma~\ref{lem:blowup}, $x=S(\varepsilon)^{-1}a$ satisfies $G(x)=0$, and
$\nrm{x-\bar x}_\nu\le\varepsilon^{-2}\rho+r_0=r_1$. Hence $x=x^*$ and $a=a^*$.
\end{proof}

\begin{proof}[Proof of Theorem~\ref{thm:A}]
By Theorem~\ref{thm:NK} and the edge version of Lemma~\ref{lem:glue}, Proposition~\ref{prop:region} gives a
locally unique family $a^{*,1}$ on $[0.9,1.1]\times[0.25,4]$, in the sense of Definition~\ref{def:family},
with the region boxes as cover. Proposition~\ref{prop:onset} gives a locally unique family $x^*$ of zeros of
$G$ in $(q,\varepsilon)$. For $\varepsilon>0$ the map $S(\varepsilon)h=\varepsilon c(h)$ sends it to a
continuous family $a^{*,2}=S(\varepsilon)x^*$ of zeros of $F$ on $\{0<r-r_c(q)\le\OnsEpsMax^2\}$.

By Lemma~\ref{lem:uniqonset} the family $a^{*,2}$ is locally unique, but only in balls of radius
$O(\varepsilon^2)$, which is the reason for the asymmetric formulation of Lemma~\ref{lem:glue}.
Because $\nrm{S(\varepsilon)h}_\nu\le\varepsilon\nrm h_\nu$, we have
$a^{*,2}\in\bar B_{\varepsilon r_0}(S(\varepsilon)\bar x)$, which is all the second family needs.

The overlap $\{(q,r)\in[0.9,1.1]\times\R: 0.25\le r\le r_c(q)+\OnsEpsMax^2\}$ is connected. At $(1,0.28)$,
with $\varepsilon_*=\sqrt{0.28}$, the hypothesis of Lemma~\ref{lem:glue} holds with $\bar y=S(\varepsilon_*)\bar x$
and $\rho=\varepsilon_*r_0$, with margin $\OnsLinkMargin$. Hence the two families coincide on the overlap. Lemma~\ref{lem:classical} turns zeros into classical solutions. Local uniqueness follows from
Proposition~\ref{prop:region} for $r\ge0.25$: a zero $a$ with $\nrm{a-a^*}_\nu\le r_1-r_0$ lies within $r_1$ of
the centre and therefore equals $a^*$. Near onset it follows from Lemma~\ref{lem:uniqonset} with
$\varepsilon^2=r-r_c(q)$. The coefficient $a^*_1$
is continuous on the connected set $\mathcal D$ and never zero, by the nontriviality margins of
Propositions~\ref{prop:region} and~\ref{prop:onset} ($a_1=\varepsilon x_1$ with $\abs{x_1}>0$ near onset). It is
positive at $(q,r)=(1,1)$, where the family contains the roll of Table~\ref{tab:points}
(Proposition~\ref{prop:region}), hence on all of $\mathcal D$. Finally,
\[
  \nrm{a^*}_\nu=\varepsilon\nrm{c(x^*)}_\nu\le\varepsilon\bigl(\nrm{\bar x}_\nu+r_0\bigr)\le\OnsMaxXnorm\,\varepsilon.\qedhere
\]
\end{proof}

\section{Spectral instability}\label{sec:spectral}

\begin{figure}[t]
\centering\includegraphics[width=0.85\linewidth]{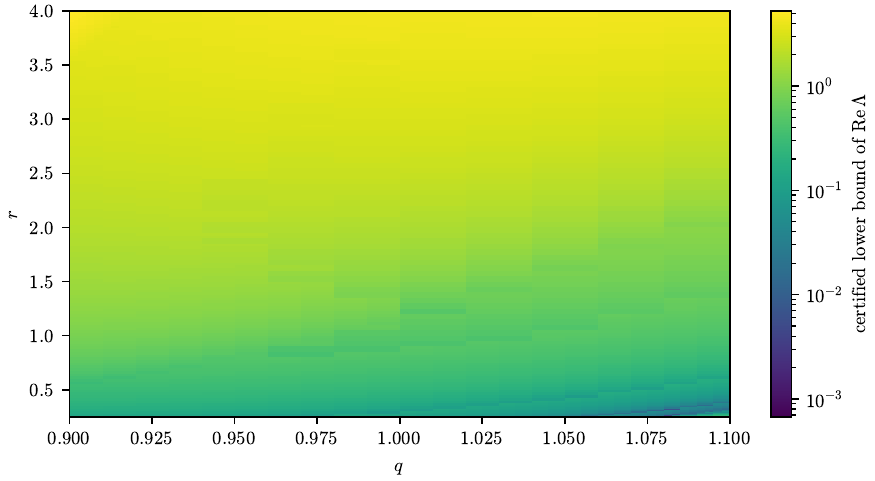}
\caption{Certified lower bound of $\Re\Lambda$ over the boxes of Proposition~\ref{prop:instab} (logarithmic
colour scale).}
\label{fig:instmap}
\end{figure}
\begin{figure}[t]
\centering\includegraphics[width=0.85\linewidth]{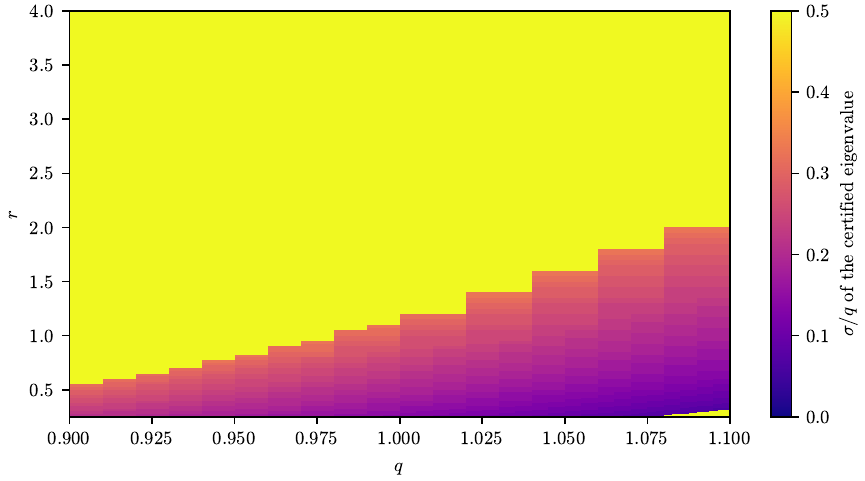}
\caption{Bloch number $\sigma/q$ of the eigenvalue certified on each box of Proposition~\ref{prop:instab}. Since
$\sigma$ is fixed on a box, the ratio is shown at the box centre $q_m$. On $\IrFracPD\%$ of the boxes
$\sigma$ is the binary64 approximation of $q_m/2$, and the selected mode corresponds numerically to the
period-doubling mode: its Floquet multiplier $e^{2\pi i\sigma/q}$ is close to, but in general not equal to,
$-1$. Elsewhere it is a finite-wavelength mode.
The selection is numerical, as the argmax over a grid of $40$ values of $\sigma$, and gives $\sigma/q_m\ge1/16$
on every box.}
\label{fig:sigmamap}
\end{figure}
\begin{figure}[t]
\centering\includegraphics[width=0.85\linewidth]{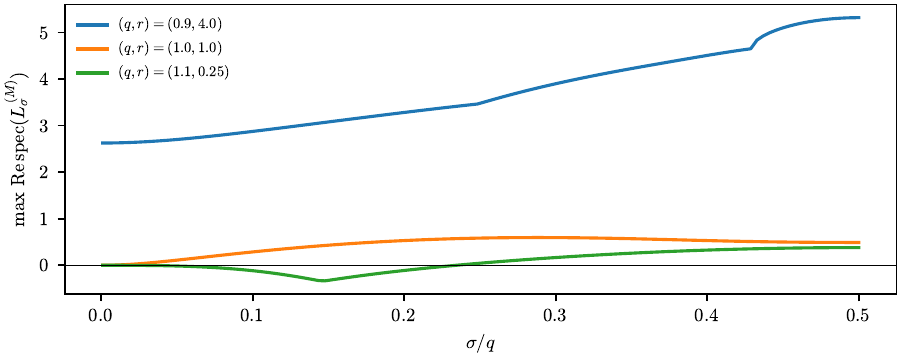}
\caption{Numerical (non-rigorous) scan of the largest real part of the truncated Bloch spectrum ($M=24$)
as a function of $\sigma/q$ at three parameter values.}
\label{fig:scan}
\end{figure}

Write $u=u_{q,r}=\sum_{m\ne0}\hat u_me^{imqx}$ with $\hat u_{\pm k}=\mp ia_k$, and let
$g(\kappa)=\kappa^2\bigl(r-(1-\kappa^2)^2\bigr)$ be the symbol of the linear part. For $\sigma\in\R$ the Bloch
ansatz $v=e^{i\sigma x}\sum_{n\in\Z}p_ne^{inqx}$ with $\kappa_n=\sigma+nq$ turns $\mathcal L v=\Lambda v$ into
\begin{equation}\label{eq:bloch}
  \bigl(g(\kappa_n)-\Lambda\bigr)p_n-i\kappa_n(\hat u*p)_n=0,\qquad n\in\Z .
\end{equation}

\begin{lemma}[Weyl sequence]\label{lem:weyl}
Let $\nu>1$, $p\in\ell^1_\nu(\Z)\setminus\{0\}$ and $\Lambda\in\C$ satisfy \eqref{eq:bloch}. Then
$\Lambda\in\operatorname{spec}_{L^2(\R)}(\mathcal L_{q,r})$.
\end{lemma}

\begin{proof}
The function $P(x)=\sum p_ne^{inqx}$ is analytic, periodic and not identically zero, and $v=e^{i\sigma x}P$ satisfies
$(\mathcal L-\Lambda)v=0$ classically. Take $\chi\in C_c^\infty(\R)$ with $\chi=1$ on $[-1,1]$ and support in
$[-2,2]$, and set $w_n=\chi(\cdot/n)v\in H^6(\R)$. By the Leibniz rule,
$(\mathcal L-\Lambda)w_n=\sum_{j=1}^6n^{-j}\chi^{(j)}(\cdot/n)D_jv$, where the $D_j$ are differential operators
with bounded periodic coefficients. Hence $\nrm{(\mathcal L-\Lambda)w_n}_{L^2}\le Cn^{-1}(4n)^{1/2}$. On the
other hand, $\abs v=\abs P$ is periodic with positive mean square, so $\nrm{w_n}_{L^2}^2\ge c\,n$ for some
$c>0$. Therefore $\nrm{(\mathcal L-\Lambda)w_n}/\nrm{w_n}\to0$, and $\mathcal L-\Lambda$ has no bounded inverse.
\end{proof}

\paragraph{Enclosure of an eigenpair.} We solve
\[
  \mathcal F(\Lambda,p)=\Bigl(p_{n_*}-1,\ \bigl((g(\kappa_n)-\Lambda)p_n-i\kappa_n(\hat u*p)_n\bigr)_{n\in\Z}\Bigr)=0
\]
in $\C\times\ell^1_\nu(\Z)$, with norm $\abs\Lambda+\sum_n\abs{p_n}\nu^{\abs n}$. The approximate inverse is
\[
  A=\operatorname{diag}\Bigl(A_{\rm fin},\bigl((g(\kappa_n)-\bar\Lambda)^{-1}\bigr)_{\abs n>M}\Bigr),\qquad M=24.
\]
Two-sided $\ell^1_\nu(\Z)$ is a Banach algebra, and $\mathcal F$ is quadratic with
\[
  D^2\mathcal F\bigl[(d\Lambda,h),(d\Lambda',h')\bigr]=-(d\Lambda\,h'+d\Lambda'\,h).
\]

\begin{lemma}[Bloch tail]\label{lem:blochtail}
Let $D(\kappa)=\kappa^2((\kappa^2-1)^2-r)-\abs{\bar\Lambda}$ and $\kappa_{\min}=(M+1)q-\abs\sigma$, and assume
$\kappa_{\min}\ge1$ and $D(\kappa_{\min})>0$. Then for $\abs n>M$,
\[
  \abs{g(\kappa_n)-\bar\Lambda}^{-1}\le D(\kappa_{\min})^{-1},\qquad
  \abs{\kappa_n}\,\abs{g(\kappa_n)-\bar\Lambda}^{-1}\le\kappa_{\min}/D(\kappa_{\min}).
\]
\end{lemma}

\begin{proof}
For $\abs n>M$ we have $\abs{\kappa_n}\ge\kappa_{\min}$ and
$\abs{g(\kappa_n)-\bar\Lambda}\ge\abs{g(\kappa_n)}-\abs{\bar\Lambda}\ge D(\abs{\kappa_n})$. The function $D$ is
increasing on $[\kappa_{\min},\infty)$. Since $D(\kappa_{\min})>0$ and
$D'\kappa-D=\kappa^2((\kappa^2-1)^2-r)+4\kappa^4(\kappa^2-1)+\abs{\bar\Lambda}>0$ there (note
$D(\kappa_{\min})>0$ forces $\kappa^2>1+\sqrt r$), the positive function $\kappa/D(\kappa)$ is decreasing.
\end{proof}

Order the unknowns as $(\Lambda,(p_n)_{\abs n\le M})$, give them weights $1$ and $\nu^{\abs n}$, and use the
induced matrix norm $\nrm M=\max_jw_j^{-1}\sum_i\abs{M_{ij}}w_i$. Let $D_\kappa=\operatorname{diag}(0,(\kappa_n)_n)$ and $K=\abs{D_\kappa}=\operatorname{diag}(0,(\abs{\kappa_n})_n)$;
the first entry is $0$ because the normalisation row carries no factor $\kappa$. Since $AD_\kappa$ and $AK$
differ only by the signs of columns, they have the same induced norm.

\begin{proposition}\label{prop:eigbounds}
Let $\bar a$ have support in $1\le k\le N$ and $\bar p$ have support in $\abs n\le M$. Let $J_{\rm fin}$ be the
Jacobian of the first $2M+2$ components of $\mathcal F$ at $(\bar\Lambda,\bar p)$, computed with $\hat{\bar u}$.
Assume that the true roll satisfies $\nrm{a^*-\bar a}_\nu\le r_0^{\rm roll}$, and assume the hypotheses of
Lemma~\ref{lem:blochtail}. Then the bounds required by Theorem~\ref{thm:NK} hold
with
\begin{align*}
  Y_0&=\nrm{A_{\rm fin}\mathcal F_{\rm fin}}+\sum_{M<\abs n\le M+N}\frac{\abs{\mathcal F_n}}{D(\kappa_{\min})}\nu^{\abs n}
     +\nrm{AK}\,2r_0^{\rm roll}\nrm{\bar p},\qquad Z_0=\nrm{I-A_{\rm fin}J_{\rm fin}},\\
  Z_1&=\bigl\|\,\abs{A_{\rm fin}}\zeta\,\bigr\|+\frac{\kappa_{\min}}{D(\kappa_{\min})}\,2\nrm{\bar a}_\nu+\nrm{AK}\,2r_0^{\rm roll},\\
  &\qquad\zeta=\bigl(0,(\zeta_n)_{\abs n\le M}\bigr),\quad\zeta_n=\abs{\kappa_n}\max_{\abs j>M,\ \abs{n-j}\le N}\abs{\hat{\bar u}_{n-j}}\,\nu^{-\abs j},\\
  Z_2&=2\max\bigl(\nrm{A_{\rm fin}}_{p},1/D(\kappa_{\min})\bigr),\qquad
  \nrm{AK}\le\max\bigl(\nrm{A_{\rm fin}K_{\rm fin}},\kappa_{\min}/D(\kappa_{\min})\bigr),
\end{align*}
where $\mathcal F_n$ is evaluated with $\hat{\bar u}$, $\nrm{\cdot}_p$ restricts to the $p$-columns, and
$\zeta_n=0$ when the index set of the maximum is empty (rows with no coupling to the tail). The
conclusion of Theorem~\ref{thm:NK} follows whenever $P(r_0)<0$ and $P(r_1)<0$. As for
$K$, the first component of $\zeta$ corresponds to the normalisation row and should not be confused with the
Fourier mode $n=0$.
\end{proposition}

\begin{proof}
The convolution $\hat{\bar u}*\bar p$ is supported in $\abs n\le M+N$, which gives the first two terms of $Y_0$, by
Lemma~\ref{lem:blochtail}. The true roll differs by $\delta\hat u$ with two-sided norm at most
$2r_0^{\rm roll}$. The resulting term $-iD_\kappa(\delta\hat u*\cdot)$ of $\mathcal F$ and of $D\mathcal F$ is
bounded through $\nrm{AD_\kappa}=\nrm{AK}$ and the Banach algebra property of two-sided $\ell^1_\nu(\Z)$. This gives the $r_0^{\rm roll}$
terms in $Y_0$ and $Z_1$.

For the remaining part of $Z_1$, write
\[
  I-A\,D\mathcal F=(I-AA^\dagger)+A(A^\dagger-D\mathcal F),\qquad
  A^\dagger=\operatorname{diag}\bigl(J_{\rm fin},(g(\kappa_n)-\bar\Lambda)_{\abs n>M}\bigr).
\]
\begin{itemize}[leftmargin=*]
\item Rows $\abs n\le M$: $(A^\dagger-D\mathcal F)h$ equals
$i\kappa_n\sum_{\abs j>M}\hat{\bar u}_{n-j}h_j$, whose modulus is at most
\[\abs{\kappa_n}\max_j\abs{\hat{\bar u}_{n-j}}\nu^{-\abs j}\cdot\sum_j\abs{h_j}\nu^{\abs j}\le\zeta_n\nrm h.\]
\item Rows $\abs n>M$: the entry is $(g(\kappa_n)-\bar\Lambda)^{-1}i\kappa_n(\hat{\bar u}*h)_n$, bounded by
Lemma~\ref{lem:blochtail} and $\nrm{\hat{\bar u}}=2\nrm{\bar a}_\nu$.
\end{itemize}
The bound $Z_2$ follows from the formula for $D^2\mathcal F$, whose components vanish in the normalisation row. The
factor $2$ is conservative: $\nrm{d\Lambda\,h'+d\Lambda'\,h}\le(\abs{d\Lambda}+\nrm h)(\abs{d\Lambda'}+\nrm{h'})$
would allow factor $1$. The computations use the factor $2$.
\end{proof}

\paragraph{Structure at $\sigma=0$.} On $2\pi/q$-periodic functions ($\sigma=0$),
$\mathcal L_{q,r}(\partial_xu)=0$ by translation invariance, and
$\mathcal L_{q,r}(1)=-\partial_x(u\cdot1)=-\partial_xu$ by Galilean invariance. So $0$ is an eigenvalue with the
Jordan chain $\{\partial_xu,-1\}$, and its algebraic multiplicity is at least $2$. Numerically it equals
$2$, and the check in Section~\ref{sec:impl} confirms $\mathcal L\,\partial_xu=0$ to $10^{-15}$.

Since $u$ is odd, the reflection $(Rv)(x)=v(-x)$ commutes with $\mathcal L_{q,r}$, and it maps Bloch number
$\sigma$ to $-\sigma$. Since $u$ is real, complex conjugation maps $(\sigma,\Lambda)$ to $(-\sigma,\bar\Lambda)$.
Hence the Bloch spectrum at each $\sigma$ is symmetric under conjugation. Numerically, the double zero splits for small
$\sigma$ into a pair of eigenvalues, and an exploratory scan in $\sigma$ suggests $\Re\Lambda=O(\sigma^2)$ and
$\Im\Lambda=O(\sigma)$ along it. Table~\ref{tab:points} reports, at the single value $\sigma=q/100$, the eigenvalue
selected from this pair (Remark~\ref{rem:lw}). Neither the splitting nor the scaling is proved here.

\paragraph{Parameter boxes for the spectrum.} On a box in $(q,r)$ both the roll and the eigenpair are taken
affine in $(s,t)$, and $\mathcal F$ is again a polynomial with acb coefficients. Where the roll and the
eigenvalue depend steeply on $q$, we also take the approximate inverse affine,
$A(s,t)=A_0-sA_0J_qA_0-tA_0J_rA_0$. With $E=I-A_0J_0$ this gives
\[
  I-A(s,t)J(s,t)=E-sA_0J_qE-tA_0J_rE+O(h^2),
\]
with $h$ as in Section~\ref{sec:boxes}, so the first-order terms are as small as the defect $E$ of $A_0$, whereas for a constant inverse they are
$-sA_0J_q-tA_0J_r$. Only the finite block $A_{\rm fin}$, or its affine version, is fixed on a box. The tail of $A$ is
$(g(\kappa_n(s))-\bar\Lambda(s,t))^{-1}$ at every parameter value, and the tail suprema of
Lemma~\ref{lem:blochtail} are taken with $q\ge q_m-h_q$, $r\le r_m+h_r$ and
$\abs{\bar\Lambda}\le\abs{\Lambda^0}+h_q\abs{\Lambda^q}+h_r\abs{\Lambda^r}$. On sub-boxes the roll enclosure
is recomputed and linked to the enclosure of Proposition~\ref{prop:region} by the edge version of
Lemma~\ref{lem:glue}, checked at the four corners, so it is the same roll. All $\IrRollLinks$ such links
hold with margin at least $\IrMinLink$, and the eigenvalue proof on each of these boxes uses exactly the
linked roll. The Bloch number $\sigma$ is a fixed binary64 number on each box, an approximation of $q_m i/80$ for the
$i\in\{1,\dots,40\}$ that maximises the largest real part of the truncated spectrum at the box centre. Its
exact value is immaterial, since Theorem~\ref{thm:B} asserts only the existence of some $\sigma$. Instability is certified on a box when
$\Re\Lambda^0-h_q\abs{\Re\Lambda^q}-h_r\abs{\Re\Lambda^r}-r_0>0$.

\begin{proposition}\label{prop:instab}
The rectangle $[0.9,1.1]\times[0.25,4]$ is covered by $\IrBoxes$ boxes on which an unstable Bloch eigenvalue is
certified. Of these, $\IrBoxesOne$ use a constant approximate inverse; they cover $\IrCoverOne\%$ of the area,
and the remaining $\IrUncovered$ boxes, all in $q\in[0.90,0.93]$, $r\in[3.5,4]$, are refined into
$\IrBoxesTwo$ boxes with the affine inverse. The smallest of the certified lower bounds is $\IrMinRe$, and every box with
$r_{\rm hi}>\IrThreshR$ has certified lower bound at least $0.13$. The smallest bound
occurs on a box near $q\approx\IrArgminQ$, $r=0.25$, with $\sigma\approx q_m/2$, where the computed eigenvalue at the box
centre is numerically real, $\Lambda\approx\IrArgminLam$.
\end{proposition}

\begin{proof}[Proof of Theorems~\ref{thm:B} and~\ref{thm:C}]
Combine Proposition~\ref{prop:instab}, Theorem~\ref{thm:NK} and Lemma~\ref{lem:weyl} for Theorem~\ref{thm:B}.
For the bound $0.13$: the boxes with $r_{\rm hi}>\IrThreshR$ are closed, finite in number and cover
$[0.9,1.1]\times(\IrThreshR,4]$, so they cover its closure $[0.9,1.1]\times[\IrThreshR,4]$.
On a box the fixed $\sigma\le q_m/2$ may exceed $q/2$ for $q<q_m$. In that case we use that Bloch numbers are
defined modulo $q$ and that conjugation maps $(\sigma,\Lambda)$ to $(-\sigma,\bar\Lambda)$. The Bloch number
$q-\sigma\in(0,q/2)$ then carries the eigenvalue $\bar\Lambda$, which has the same real part. So $\sigma$ can
always be normalised to $(0,q/2]$ as stated.
For Theorem~\ref{thm:C} use the point proofs summarised in Table~\ref{tab:points}; their rolls belong to the
family of Theorem~\ref{thm:A} by Proposition~\ref{prop:region}.
\end{proof}

\begin{table}[t]
\centering\footnotesize\setlength{\tabcolsep}{3.5pt}
\begin{tabular}{llrrrlrr}
\toprule
$q$ & $r$ & $\bar a_1$ & roll radius & $\sigma^*$ & $\Lambda$ (most unstable) & radius & $\Re\Lambda_{\rm lw}$ at $\sigma=q/100$\\
\midrule
\csname @@input\endcsname table_points.tex
\bottomrule
\end{tabular}
\caption{Certified data at nine parameter values. Roll radius: $r_0$ of the existence proof in $\lnu$.
The Bloch number $\sigma^*$ maximises the largest real part of the truncated spectrum over a grid in
$(0,q/2]$. Centres are rounded to four decimals; the radii refer to the unrounded centres in the run data,
so each printed $\bar a_1$ and $\Lambda$ is within $10^{-4}$ of the true value.
The last column gives the centre of the enclosure of $\Re\Lambda_{\rm lw}$ to two significant digits; its
sign is certified, since the radii are below $3\cdot10^{-13}$.}
\label{tab:points}
\end{table}

\section{Implementation and verification}\label{sec:impl}

All bounds are evaluated in midpoint-radius interval arithmetic with Arb \cite{Johansson2017}, through
python-flint 0.9.0, at 212 bits.
\begin{itemize}[leftmargin=*]
\item \emph{Inputs and centres.} Parameters, the Bloch number $\sigma=q/100$ of Theorem~\ref{thm:C}, and
the points at which gluing inequalities and links between enclosures are checked are entered as decimal strings and
enclosed in balls, so every statement holds for the exact decimal values. Floating-point centres and approximate inverses are converted
exactly; their accuracy affects only the size of the bounds, never their validity.
\item \emph{Galerkin truncation.} $N=32$ for rolls. For Bloch modes, $M=24$ on the parameter boxes of
Proposition~\ref{prop:instab} and $M=48$ in the point proofs of Theorem~\ref{thm:C}. Each proof is complete at
its own $M$. The smaller value on the boxes was chosen for speed, guided by an exploratory, unrecorded
comparison with $M=48$ on two test boxes; no statement of the paper depends on that comparison.
\end{itemize}

\paragraph{Predictor and verifier.} Each computation has two stages.
\begin{itemize}[leftmargin=*]
\item \emph{Predictor} (float64, not part of the proofs): centres, tangents, approximate inverses and the
choice of $\sigma$.
\item \emph{Verifier} (Arb): takes these as exact binary numbers and checks the hypotheses of
Theorem~\ref{thm:NK} and Lemma~\ref{lem:glue}.
\end{itemize}
The predictor affects only whether a verification succeeds; the validity of a successful verification
depends on the verifier alone.
Table~\ref{tab:constants} collects the worst constants of the three coverings.

\begin{table}[t]
\centering\scriptsize\setlength{\tabcolsep}{4.5pt}
\begin{tabular}{llrrrrrrrrr}
\toprule
covering & trunc. & boxes & gluings & $\max Y_0$ & $\max Z_0$ & $\max Z_1$ & $\max(Z_0{+}Z_1)$ & $\max r_0$ & $\min r_1$ & time [s]\\
\midrule
\csname @@input\endcsname table_constants.tex
\bottomrule
\end{tabular}
\caption{Worst constants over each covering. Upper bounds are rounded up and lower bounds down. Wall-clock
times are on an 8-core laptop.}
\label{tab:constants}
\end{table}

Independent checks, which do not enter the proofs:
\begin{itemize}[leftmargin=*]
\item the identity behind \eqref{eq:F} and the rescaling of Lemma~\ref{lem:blowup}, checked exactly with SymPy;
\item agreement of the Jacobians and tangents with finite differences;
\item the residual of the equation in physical space, with the product evaluated pointwise;
\item the double zero eigenvalue at $\sigma=0$ with eigenvector $u_x$;
\item agreement of the eigenvalue between $M=48$ and $M=96$ (Newton refinement);
\item sampled residuals stay below the box bounds (at most $55\%$ of $Y_0$ and $79\%$ of $Z_0$);
\item exact reproduction of an existing proof when a box has zero width;
\item regression against a Kuramoto--Sivashinsky computation obtained with the same code.
\end{itemize}

All computations used in the proofs are recorded runs with hashed code, inputs and outputs.

\section{Discussion}\label{sec:disc}

Theorem~\ref{thm:A} establishes a continuum of finite-amplitude rolls in a band of wavenumbers, connected to
the trivial state at onset. The closed-form periodic solutions discussed in Remark~\ref{rem:exact} occur at the special value
$r=81/121$. Theorem~\ref{thm:B} gives a rigorous, quantitative form of the instability predicted in
\cite{TribelskyVelarde1996,MatthewsCox2000} on a large parameter region far from onset.
Theorem~\ref{thm:C} suggests that, away from onset, the long-wave branch need not carry the instability: at
$(q,r)=(0.9,4)$ and $(1.1,0.25)$ the eigenvalue numerically identified with this branch has negative real part
at $\sigma=q/100$, while eigenvalues at larger Bloch numbers have $\Re\Lambda\ge\InstMinRe$. Whether the branch
is stable for all small $\sigma$ is open (Remark~\ref{rem:lw}). The uniform bound $\IrMinRe$ of
Theorem~\ref{thm:B} is the smallest bound over the boxes, not an estimate of the weakest growth rate: on
the box where it occurs, the eigenvalue at the centre is about $\IrArgminLam$.
Figure~\ref{fig:sigmamap} shows two regimes of the numerically selected most unstable mode. Above a curve
rising from $r\approx0.55$ at $q=0.9$ to $r\approx2$ at $q=1.1$, the period-doubling mode dominates; below
it, a finite-wavelength mode does. The smallest bound occurs next to a small period-doubling patch near
$(1.08,0.25)$, where the selected mode changes type. The step shape of the curve
is an artefact of the box resolution, and its location is a numerical observation.

\paragraph{Limitations and open problems.}
\begin{enumerate}[leftmargin=*]
\item \emph{Instability near onset.} Spectral instability for $r_c(q)<r<0.25$ is not covered. For fixed
$q\ne1$ the roll tends to $0$ as $r\downarrow r_c(q)$, but the limiting zero state is itself unstable:
its symbol satisfies $g(1)=r_c(q)>0$, and $\kappa_1=1$ for the Bloch number $\sigma=1-q$. Instability for $r$ close to $r_c(q)$ should then follow by
perturbation, with a range of validity that shrinks as $q\to1$. Near $(q,r)=(1,0)$ the critical spectral values
approach zero, which makes uniform instability estimates more delicate; a uniform treatment, for instance a desingularised Bloch problem in the Matthews--Cox scaling, is
needed.
\item \emph{Other wavenumbers.} The band $q\in[0.9,1.1]$ was chosen for definiteness. The method extends
to other bands as long as the nondegeneracy of Lemma~\ref{lem:nondeg} and the tail conditions persist. It
fails at $q=\sqrt{2/5}$, where $\lambda_2(q,r_c(q))=12q^4(5q^2-2)$ vanishes and the second mode is also
critical.
\item \emph{Nonlinear instability} is not addressed.
\item \emph{The long-wave branch.} Theorem~\ref{thm:C} encloses an eigenvalue at a single Bloch number.
Identifying it with the branch from the double zero, and determining the sign of its real part for all small
$\sigma$, would require a validated enclosure of the spectral subspace near $0$, uniformly in $\sigma$.
\end{enumerate}

\section*{Data availability}
The code for the computer-assisted proofs, the records of all computational runs (hashes of the code,
inputs and outputs), and a script that regenerates every number, table and figure of this paper from these
records are provided as supplementary material.

\section*{Declarations}
\paragraph{Use of AI tools.} In accordance with SIAM's policy on artificial intelligence, we report that
the AI assistant Claude (Anthropic) was used in the following parts of this work: literature search and
drafting of the discussion of related work, including Remark~\ref{rem:exact}; drafting of statements and
proofs; development and testing of the interval-arithmetic code and execution of the recorded computations
(Section~\ref{sec:impl}); the scripts that generate the figures and tables; and collection of bibliographic
data, which were checked against Crossref and arXiv metadata. All figures are produced by code, and every
computation can be reproduced from the supplementary material. The authors defined the objectives of the
work, directed it, checked the mathematical content, tested the code, and read the cited works.
The authors assume responsibility for all content.

\bibliographystyle{plainnat}
\bibliography{refs}

\end{document}